\documentclass[reqno]{amsart}

\usepackage{hyperref}

\newtheorem{theorem}{Theorem}[section]
\newtheorem{lemma}[theorem]{Lemma}
\newtheorem{corollary}[theorem]{Corollary}

\theoremstyle{definition}

\theoremstyle{remark}
\newtheorem{remark}[theorem]{Remark}

\newtheorem{claim}{Claim}

\numberwithin{equation}{section}

\newcounter{casenum}

\usepackage{amssymb}
\renewcommand{\leq}{\leqslant}
\renewcommand{\geq}{\geqslant}

\renewcommand{\{}{\left\lbrace}
\renewcommand{\}}{\right\rbrace}
\newcommand{\floor}[1]{{\left\lfloor#1\right\rfloor}}

\newcommand{\pfrac}[2]{{\left(\frac{#1}{#2}\right)}}
\newcommand{\ptext}[1]{{\left(\textup{#1}\right)}}

\newcommand{\op}[1]{\operatorname{#1}}
\def\GL{\op{GL}}
\def\SL{\op{SL}}
\def\PGL{\op{PGL}}
\def\PSL{\op{PSL}}

\def\F{\mathbf{F}}
\def\P{\mathbf{P}}
\def\E{\mathbf{E}}
\def\Var{\op{Var}}
\def\frakC{\mathfrak{C}}
\def\ord{\op{ord}}
\def\gcd{\op{gcd}}
\def\lcm{\op{lcm}}

\def\harm{\eta}
\def\loglog{\log\log}

\usepackage{todonotes}

\begin{document}

\title{Random generation of the special linear group}

\author{Sean Eberhard}
\address{Sean Eberhard, London, UK}
\curraddr{}
\email{eberhard.math@gmail.com}

\author{Stefan-C. Virchow}
\address{Stefan-C. Virchow, University of Rostock, Rostock, Germany}
\email{stefan.virchow@uni-rostock.de}

\subjclass[2000]{}




\begin{abstract}
It is well known that the proportion of pairs of elements of $\SL(n,q)$ which generate the group tends to $1$ as $q^n\to \infty$. This was proved by Kantor and Lubotzky using the classification of finite simple groups. We give a proof of this theorem which does not depend on the classification.

An essential step in our proof is an estimate for the average of $1/\ord g$ when $g$ ranges over $\GL(n,q)$, which may be of independent interest. We prove that this average is
\[
  \exp(-(2-o(1)) \sqrt{n \log n \log q}).
\]
\end{abstract}

\maketitle

\section{Introduction}

Nonabelian finite simple groups $G$ are generated by a generic pair of elements $x,y\in G$, i.e.,
\[
  \frac1{|G|^2} |\{(x,y) \in G^2 : \langle x,y\rangle = G\}| \longrightarrow 1
\]
as $|G|\to\infty$, the limit being taken along any sequence of finite simple groups $G$. This was originally conjectured by Dixon~\cite{dixon} (before we even knew that every finite simple group is $2$-generated). It was proved for the alternating groups by Dixon in the same paper, for classical groups by Kantor and Lubotzky~\cite{kantor-lubotzky}, and finally for exceptional groups by Liebeck and Shalev~\cite{liebeck--shalev--95}.

Almost all of this work depends on the classification of finite simple groups (CFSG). To some extent this is natural and unavoidable: how can one be expected to prove a theorem about finite simple groups without even knowing what they are? But for specific families such as $\SL(n, 2)$ it is lamentable. In this connection we echo an opinion of Kantor (see \cite[Preface]{cameron-kantor-2018}): the classification should not be invoked when it is not needed. We believe that a proof not depending on CFSG, where possible, is often more illuminating, and more generalizable.

We are aware of CFSG-free proofs of Dixon's conjecture only for $A_n$ and for groups of essentially bounded rank. For $A_n$, see Dixon's original proof, improvements by Bovey and Williamson~\cite{bovey-williamson} and Bovey~\cite{bovey}, or our sharp estimate \cite{eberhard-virchow-Sn}. For groups of bounded rank (or, at best, rank bounded by a rather slowly growing function of the characteristic), there is a large body of work culminating in the theorem that a random pair of elements not only generates but defines an expander Cayley graph: see Breuillard, Green, Guralnick, and Tao~\cite{bggt}. The critical group-theoretic input for this theorem is the Larsen--Pink theorem~\cite{larsen-pink}, which relates general finite subgroups of bounded-rank algebraic groups to algebraic or arithmetic subgroups: this theorem relies on the theory of algebraic groups, not on the classification. In all other cases we are not aware of a CFSG-free proof.

The purpose of this note is to supply such a proof for the family of special linear groups $\SL(n,q)$. We prove the following theorem.

\begin{theorem}\label{main-thm}
Let $\pi,\sigma \in \SL(n,q)$ be chosen uniformly at random. Then
\[
  \P(\langle \pi, \sigma\rangle \neq \SL(n,q)) \leq e^{-c \sqrt{n} \log q} + e^{-(2-o(1)) \sqrt{n \log n \log q}}.
\]
\end{theorem}

Which of the two terms in the theorem dominates depends on the relative size of $n$ and $q$, but neither is sharp in any case, since the Kantor--Lubotzky proof shows (see Kantor~\cite[Theorem~3.3]{kantor-some-topics-survey-92}) that
\[
  \P(\langle \pi, \sigma\rangle \neq \SL(n,q)) = 2q^{-n} + O(q^{-7(n-1)/6}).
\]
We have not been able to prove this sharp estimate with our method.

To give a quick sketch of the proof, the main idea is to use the ``$\pi\sigma^i$ trick'': if $\pi,\sigma\in G$ are random, then the elements $\pi, \pi\sigma, \dots, \pi\sigma^{N-1}$ should be approximately pairwise equidistributed, so we ought to be able to use the second moment method to prove that there must be some $i$ such that $\pi\sigma^i \in \frakC$. Here $\frakC$ is any set we like which is both conjugation-invariant and large enough. Bounding the variance in the second moment method reduces to bounding some character sums, and for this we rely on character estimates of Larsen, Shalev, and Tiep~\cite{larsen-shalev-tiep} (and the strength of these estimates is represented in the first term in Theorem~\ref{main-thm}).

Next we need to choose a few large conjugacy-invariant subsets $\frakC_i \subset G$ such that if $g_i \in \frakC_i$ for each $i$ then $\langle g_1, g_2, \dots \rangle = G$. We take $\frakC_1$ to be the set of all irreducible $g\in \SL(n,q)$ of order $(q^n-1)/(q-1)$, and $\frakC_2$ to be the set of all $g\in \SL(n,q)$ of order $q^{n-1}-1$ preserving a decomposition $\F_q^n = \ell \oplus W$ with $\dim \ell=1$ and $\dim W = n-1$. We can prove that $\langle g_1, g_2\rangle = G$ under these circumstances by quoting a well-known theorem of Cameron and Kantor on multiply transitive subgroups of $\PGL(n,q)$.

Interestingly, an essential ingredient in our method is an estimate for the average of $1/\ord g$ for $g \in G$, or equivalently the harmonic mean of the orders of the elements of $\SL(n,q)$. Because it might be of independent interest, we devote some effort to obtaining a sharp estimate.

\begin{theorem}
Let
\[
  \harm_G = \frac1{|G|} \sum_{g \in G} \frac1{\ord g} .
\]
If $G$ is $\GL(n,q)$, $\PGL(n,q)$, $\SL(n,q)$, or $\PSL(n,q)$, then
\[
  \harm_G \leq \exp(-(2-o(1)) \sqrt{n \log n \log q}).
\]

Moreover, if $n$ is sufficiently large compared to $q$ then
\[
  \harm_G = \exp(-(2+o(1)) \sqrt{n \log n \log q}).
\]
\end{theorem}

In fact, in general $\eta_G$ represents a bottleneck in our method. This bottleneck already featured in our previous work \cite{eberhard-virchow-Sn}, in which we applied the same method to the alternating group to prove, independent of the classification, that two elements $\pi, \sigma \in A_n$ will generate with probability $1 - 1/n + 1/n^{2+o(1)}$. The pithy reason that we were not able to push beyond the $1/n^{2+o(1)}$ term is that $\eta_{A_n} = 1/n^{2+o(1)}$.

Finally, we believe our method may generalize to other finite simple groups of Lie type. We can continue to apply the Larsen--Shalev--Tiep character estimates, so to proceed one would need two main ingredients:
\begin{enumerate}
    \item a few large conjugacy-invariant subsets $\frakC_i$ such that for any choice of $g_i \in \frakC_i$ for each $i$ we have $\langle g_1, g_2, \dots \rangle = G$;
    \item an estimate for $\eta_G$.
\end{enumerate}
We hope to return to this challenge in future work.

\subsection{Notation}

We have already introduced our most nonstandard convention, which is the symbol
\[
  \harm_G = \frac1{|G|} \sum_{g \in G} \frac1{\ord G}
\]
for any group $G$. (The intention is that $\eta$ stands for ``harmonic''.)

\def\Irr{\op{Irr}}

We denote by $\Irr(G)$ the set of irreducible characters of $G$, and we write $1$ for the trivial character. We also write $1$ for the identity element of $G$. Given a set $\frakC \subset G$ we write $1_\frakC$ for the indicator of $\frakC$, so for instance
\[
  \langle \chi, 1_\frakC \rangle = \frac1{|G|} \sum_{g \in \frakC} \chi(g) \qquad (\chi \in \Irr(G)).
\]

We use standard big-$O$ and little-$o$ notation all over the shop. Our convention is that $f(x) = O(g(x))$ means $|f(x)| \leq C g(x)$ for some constant $C$ and for all $x$ under consideration (i.e., no specific limit $n\to\infty$ or $q\to\infty$ is assumed). When we write $o(g)$ we mean an estimate which holds in either limit $n\to\infty$ or $q\to\infty$, unless otherwise specified. In one or two places we use the Vinogradov notation $f \ll g$. This means simply $f = O(g)$.

\def\totient{\varphi}
\def\divisor{\sigma_0}

We write $\divisor(n)$ for the divisor-counting function, $\totient(n)$ the Euler totient function, and $p(n)$ for the partition function. We may as well note the following basic bounds now:
\begin{align}
    \divisor(n) &\leq e^{O(\log n / \loglog n)}, \label{divisor-bound} \\
    \totient(n) &\gg \frac{n}{\loglog n}, \label{totient-bound} \\
    p(n) &\leq e^{O(n^{1/2})}. \label{partition-bound}
\end{align}
See Hardy--Wright~\cite{hardy-wright} for the first two of these (see Theorems~317 and 328, respectively). For \eqref{partition-bound}, see \cite{hardy--ramanujan}.

\subsection{Acknowledgements}

The first author is grateful to Bill Kantor for remembering and discussing the technical details of a 35-year-old paper. We are also grateful to Will Sawin and Felipe Voloch, who in answering a Mathoverflow question provided us with the proof of an essential lemma (see the appendix). The second author is grateful to Jan-Christoph Schlage-Puchta for the many inspiring discussions we had.

\section{The \texorpdfstring{$\pi\sigma^i$}{pi sigma\^{}i} trick}\label{sec:pisigma}

By ``$\pi \sigma^i$ trick'' we mean the observation that the elements $\pi \sigma^i$ for $0 \leq i < N$ should behave roughly pairwise independently, combined with the second moment method. To our knowledge this idea was first made explicit by Babai, Beals, and Seress~\cite[Section~4]{babai--beals--seress}, and it has been used to good effect several times since (see, e.g.,~\cite{babai--hayes, Schlage-Puchta2012, helfgott--seress--zuk}). We use it to prove the following theorem.

\def\chineq1{{\substack{\chi \in \Irr(G) \\ \chi \neq 1}}}
\def\chietaratio{{\frac{\langle \chi, \eta \rangle}{\chi(1)}}}

\begin{theorem} \label{thm:pisigma}
Let $G$ be a finite group, and let $\eta$ be the class function defined by
\[
  \eta(g) = \sum_{\substack{
    \sigma \in G \\
    g \in \langle \sigma\rangle
    }}
  \frac1{\ord \sigma}.
\]
Let $\frakC$ be any conjugacy-invariant subset of $G$. If $\pi, \sigma$ are chosen uniformly at random, then
\[
  \P(\langle \pi, \sigma\rangle \cap \frakC = \emptyset)
  \leq \frac{|G|}{|\frakC|} \max_\chineq1 \chietaratio.
\]
\end{theorem}

\begin{proof}(cf.~\cite[Section~3]{eberhard-virchow-Sn})
Let $N$ be a positive integer and let $X_N$ be the number of $i \in \{0, \dots, N-1\}$ such that $\pi\sigma^i \in \frakC$. By Chebyshev's inequality,
\begin{equation} \label{chebyshev}
  \P(\langle \pi, \sigma \rangle \cap \frakC = \emptyset) \leq \P(X_N = 0) \leq \frac{\Var X_N}{(\E X_N)^2}.
\end{equation}
Clearly $\E X_N = N \frac{|\frakC|}{|G|}$, while
\begin{align*}
  \E X_N^2
  &= \sum_{i,j=0}^{N-1} \P(\pi\sigma^i, \pi\sigma^j \in \frakC) \\
  &= \sum_{i,j=0}^{N-1} \P(\pi, \pi\sigma^{j-i} \in \frakC) \\
  &= \sum_{i=1}^N 2(N-i) \P(\pi, \pi \sigma^i \in \frakC) + O(N).
\end{align*}
Expand the inside term as
\[
  \P(\pi, \pi\sigma^i \in \frakC) = \frac1{|G|^2} \sum_{\pi, \sigma} 1_\frakC(\pi) 1_\frakC(\pi \sigma^i) = \frac1{|G|^2} \sum_{\pi, g} 1_\frakC(\pi) 1_\frakC(g) r_i(\pi^{-1} g),
\]
where $r_i(x)$ is the number of $\sigma \in G$ such that $\sigma^i = x$. Note
\[
  \sum_{i = 1}^N 2(N-i) r_i(x) \sim N^2 \eta(x) \qquad (\text{as}~N\to\infty).
\]
Indeed, the left-hand side is
\[
  \sum_{\sigma\in G} \sum_{i=1}^N 2(N-i) 1_{\sigma^i = x},
\]
and the inner sum here is either zero if $x \notin \langle \sigma\rangle$ or, if $\sigma^{i_0} = x$ say, then
\[
  \sum_{\substack{1 \leq i \leq N \\ i \equiv i_0 \pmod{\ord \sigma}}} 2(N-i) \sim \frac{N^2}{\ord \sigma}.
\]
Thus
\[
  \E X_N^2 \sim \frac{N^2}{|G|^2} \sum_{\pi, g} 1_\frakC(\pi) 1_\frakC(g) \eta(\pi^{-1} g).
\]
Applying \cite[Proposition~9.33]{curtis-reiner}, this is the same as
\[
  N^2 \sum_{\chi\in\Irr(G)} \frac{|\langle \chi, 1_\frakC\rangle|^2 \langle \chi, \eta\rangle}{\chi(1)}.
\]
Note that the $\chi = 1$ term is exactly $(\E X_N)^2$. Thus in the limit $N \to \infty$ we get
\[
  \frac{\Var X_N}{(\E X_N)^2} \to \pfrac{|G|}{|\frakC|}^2 \sum_\chineq1 \frac{|\langle \chi, 1_\frakC \rangle|^2 \langle \chi, \eta \rangle}{\chi(1)}.
\]
Thus, taking the limit $N\to\infty$ in \eqref{chebyshev}, we have
\[
  \P(\langle \pi, \sigma\rangle \cap \frakC = \emptyset)
  \leq \pfrac{|G|}{|\frakC|}^2 \sum_\chineq1 \frac{|\langle \chi, 1_\frakC \rangle|^2 \langle \chi, \eta \rangle}{\chi(1)}.
\]
Finally, by orthogonality of characters we have
\[
  \sum_{\chi \in \Irr(G)} |\langle \chi, 1_\frakC \rangle|^2 = \frac{|\frakC|}{|G|}.
\]
Thus
\[
  \P(\langle \pi, \sigma\rangle \cap \frakC = \emptyset)
  \leq \frac{|G|}{|\frakC|} \max_\chineq1 \chietaratio ,
\]
as claimed.
\end{proof}

To apply the theorem we will need to bound $\langle \chi, \eta\rangle / \chi(1)$. In particular, considering the contribution from just $g=1$, we will need to bound
\[
  \frac{\eta(1)}{|G|} = \frac1{|G|} \sum_{g \in G} \frac1{\ord g}.
\]
We turn to this in the next section.

\section{The average of \texorpdfstring{$1/\ord g$}{1 / ord g}}

In this section we are concerned with bounding
\[
  \harm_G = \frac{\eta(1)}{|G|} = \frac1{|G|} \sum_{g \in G} \frac1{\ord g}
\]
for various groups $G$, particularly $\GL(n, q)$, $\PGL(n,q)$, and $\PSL(n,q)$.

\begin{theorem} \label{inverse-orders-theorem}
Let $G$ be $\GL(n, q)$, $\PGL(n,q)$, $\SL(n,q)$, or $\PSL(n,q)$. Then
\[
  \harm_G \leq \exp(-(2-o(1)) \sqrt{n \log n \log q}).
\]
Moreover, for fixed $q$ and $n \to \infty$ we have a matching lower bound
\[
  \harm_G \geq \exp(-(2+o(1)) \sqrt{n \log n \log q}).
\]
\end{theorem}

\begin{remark}
The estimate should be compared with results of Stong~\cite{stong-average-order} and Schmutz~\cite{schmutz-typical-order} for $G = \GL(n,q)$ when $n$ is large compared to $q$. Stong proved that $\ord g$ is on average $q^n / n^{1+o(1)}$, while Schmutz proved that $\ord g$ is typically $q^{n - (\log n)^{2+o(1)}}$. Theorem~\ref{inverse-orders-theorem} asserts that the harmonic mean of $\ord g$, which is precisely $\eta_G^{-1}$, is $e^{(2+o(1))\sqrt{n \log n \log q}}$. In other words, we have
\begin{align*}
    \log \ptext{mean} &= n \log q - \log n + o(\log n), \\
    \log \ptext{median} &= n \log q - (\log n)^{2 + o(1)} \log q , \\
    \log \ptext{harmonic mean} &= (2+o(1))\sqrt{n \log n \log q}.
\end{align*}
It is rather striking how much smaller the harmonic mean is.
\end{remark}

\subsection{Basic observations about \texorpdfstring{$\harm_G$}{eta\_G}}

We collect here a few basic observations about $\harm_G$ for general groups $G$.

\begin{lemma} \label{subgroup-lemma}
If $H$ is a subgroup of $G$ then $\harm_H \leq [G: H] \, \harm_G$.
\end{lemma}

\begin{lemma} \label{quotient-lemma}
If $K$ is a quotient of $G$ then $\harm_G \leq \harm_K \leq \frac{|G|}{|K|} \, \harm_G$.
\end{lemma}

The combination of the above two lemmas enables us to almost completely restrict attention to $\GL(n,q)$.

\begin{lemma} \label{cclreps}
Let $g_1, \dots, g_k$ be a complete set of conjugacy class representatives in $G$. Then
\[
  \harm_G = \sum_{i=1}^k \frac{1}{|C_G(g_i)|} \frac{1}{\ord g_i}.
\]
\end{lemma}

\begin{lemma} \label{cyclic-group}
If $C_n$ is the cyclic group of order $n$ then
\[
  \harm_{C_n} = \frac1n \sum_{d \mid n} \frac{\totient(d)}{d} \leq \frac{\divisor(n)}{n}.
\]
\end{lemma}

\subsection{Lower bound}

Let $G = \GL(n,q)$. Let $d \in [1, n]$ be an integer, and assume first for simplicity that $d$ divides $n$. Consider the contribution to $\harm_G$ from just those $g \in G$ whose characteristic polynomial splits into $n/d$ distinct irreducible factors of degree $d$. In other words, we are considering just those $g$ which are diagonalizable over $\F_{q^d}$ with distinct eigenvalues, each of which has degree $d$ over $\F_q$. Each such $g$ has order at most $q^d-1$, and
\[
  |C_G(g)| = (q^d - 1)^{n/d} < q^n.
\]
Since there are $q^d /d - O(q^{d/2} / d)$ irreducible polynomials of degree $d$, by Lemma~\ref{cclreps} the total contribution to $\eta_G$ is at least
\begin{align*}
  \binom{q^d / d - O(q^{d/2}/d)}{n/d} \frac{1}{q^{n + d}}
  &\geq \left(\frac{q^d / d - O(q^{d/2}/d)}{n/d} \right)^{n/d} \frac{1}{q^{n + d}} \\ 
  &= \frac1{n^{n/d} q^d} \left( 1 - O(q^{-d/2}) \right)^{n/d} \\
  &= \frac1{n^{n/d} q^d} e^{-O((n/d) q^{-d/2})}.
\end{align*}
To optimize $n^{-n/d} q^{-d}$ we should take
\[
  d \sim \sqrt{n \log n / \log q},
\]
and the error term is clearly negligible for this choice, so we get
\[
  \eta_G \geq \exp(-(2+o(1)) \sqrt{n \log n \log q}).
\]
This proves the lower bound whenever $n$ has a divisor $d$ of the right size.

In general $n$ need not have such a divisor, so we have to tweak the construction slightly. Let $r = n - \floor{n/d} d$, and assume that $d$ has a divisor of size
\[
  d_1 \sim n^{1/10} / \log q,
\]
let $r_1 = r - \floor{r/d_1} d_1$, and consider the contribution from those $g \in G$ whose characteristic polynomial splits into
\begin{enumerate}
    \item $\floor{n/d}$ distinct irreducible factors of degree $d$,
    \item $\floor{r/d_1}$ distinct irreducible factors of degree $d_1$,
    \item $r_1$ copies of the linear factor $X-1$.
\end{enumerate}
Still we have $\ord g \leq q^d-1$, and
\[
    |C_G(g)| = (q^d - 1)^\floor{n/d} (q^{d_1} - 1)^{\floor{r/d_1}} |\GL(r_1, q)| < q^{n - r_1 + r_1^2}.
\]
Thus the contribution to $\eta_G$ is at least
\begin{align*}
  \binom{q^d / d - O(q^{d/2}/d)}{\floor{n/d}} & \binom{q^{d_1} / d_1 - O(q^{d_1/2}/d_1)}{\floor{r/d_1}} \frac{1}{q^{n - r_1+r_1^2 + d}} \\
  & \gg \frac{q^{\floor{n/d} d}}{n^{\floor{n/d}}} \frac{q^{\floor{r/d_1} d_1}}{r^{\floor{r/d_1}}} \frac1{q^{n-r_1+r_1^2+d}} \\
  &= \frac1{n^{n/d} q^d} \exp(-O((r/d_1) \log r + r_1^2 \log q)) \\
  &\geq \frac1{n^{n/d} q^d} \exp(-O((d/d_1) \log n + d_1^2 \log q)).
\end{align*}
The error term is negligible, so this proves the lower bound.

\subsection{Some background}

\def\frakF{\mathfrak{F}}
\def\calP{\mathcal{P}}

We need to recall some basic theory about conjugacy classes in $\GL(n,q)$, some of which we already touched upon in the previous subsection. Let $\frakF$ denote the set of monic irreducible polynomials over $\F_q$, apart from $X$, and let $\frakF_d$ be set of $f\in \frakF$ of degree $d$. Any $g \in \GL(n,q)$ makes $V = \F_q^n$ into an $\F_q[X]$-module with $X$ acting as $g$. By the structure theorem for finitely generated modules over a PID, $V$ decomposes as a direct sum
\[
  V = \F_q[X] / (f_1^{m_1}) \oplus \cdots \oplus \F_q[X] / (f_l^{m_l}),
\]
where each $f_i \in \frakF$, and $f_1^{m_1}, \dots, f_l^{m_l}$ are uniquely determined up to order. The upshot is that specifying a conjugacy class in $\GL(n,q)$ is equivalent to specifying a multiset of pairs $(f_i, m_i)$, where $f_i \in \frakF$ and $m_i \geq 1$, such that $\sum_i m_i \deg f_i = n$.\footnote{Another way of expressing this is to say that conjugacy classes in $\GL(n,q)$ correspond to maps $\nu$ which assign to each $f \in \frakF$ some partition of some positive integer, such that $\sum_{f \in \frakF} |\nu(f)| \deg f = n$ (where $|\lambda|$ denotes the total size of the partition $\lambda$).}

Over $\overline{\F}_q$ this leads to the Jordan normal form of $g$, which gives us a formula for the order of $g$ in terms of the invariants $f_1^{m_1}, \dots, f_l^{m_l}$. For $f \in \frakF$, write $\ord f$ for the multiplicative order of any root of $f$. Then
\begin{equation} \label{order-formula}
  \ord g = \lcm(\ord f_1, \dots, \ord f_l) p^t,
\end{equation}
where $p$ is the prime dividing $q$ and $p^t$ is the smallest power of $p$ such that $p^t \geq \max(m_1, \dots, m_l)$.

For all $g \in G$ we have $|C_G(g)| \gg q^n / n^2$ (Fulman--Guralnick~\cite[Theorem~6.4]{fulman--guralnick}), so by the above correspondence and Lemma~\ref{cclreps} we have
\[
  \harm_G \ll \frac{n^2}{q^n} \sum_{\substack{f_1^{m_1}, \dots, f_l^{m_l} \\ \sum m_i \deg f_i = n }}
  \frac1{\lcm(\ord f_1, \dots, \ord f_l)},
\]
it being understood that the sum extends over all choices of $f_1^{m_1}, \dots, f_l^{m_l}$ up to order.

Our strategy for bounding this sum is to combine two special cases. Let $\harm_1$ be the restriction of the above sum to the case in which $m_i = 1$ for each $i$:
\[
  \harm_1 = \sum_{\substack{f_1, \dots, f_l \\ \sum \deg f_i = n}} \frac1{\lcm(\ord f_1, \dots, \ord f_l)}.
\]
At the opposite extreme, consider the case in which $m_i \geq 2$ for each $i$. Let $\harm_2$ be the sheer number of terms in that case:
\[
  \harm_2 = \sum_{\substack{f_1^{m_1}, \dots, f_l^{m_l} \\ m_i \geq 2, \sum m_i \deg f_i = n}} 1.
\]
We will bound $\harm_1$ and $\harm_2$ using two different arguments, and then we will bound $\harm_G$ via (with obvious notation)
\begin{equation} \label{harm-splitting}
  \harm_G \ll \frac{n^2}{q^n} \sum_{n_1+n_2 = n} \harm_1(n_1) \harm_2(n_2).
\end{equation}

\subsection{A lemma about random permutations}

In this subsection, for the moment something of a non sequitur, we prove the following lemma about random permutations. This will turn out to be an essential ingredient in the next subsection, in which we bound $\harm_1$.

\begin{lemma} \label{nonsequitur}
Let $q$ be a prime power, let $m$ be a positive integer, and let $\pi \in S_n$ be a random permutation. Then the probability that every cycle length of $\pi$ is contained in $D_m = \{d \leq n: q^d - 1 \mid m\}$ is bounded by $n^{-n/\log_q (m+1)} e^{O(n/\log_q m + n^{1/2})}$.
\end{lemma}

For any $d$, the probability that every cycle of $\pi$ has length at most $d$ is bounded by $1/\floor{n/d}! \approx (n/d)^{-n/d}$. The point of the lemma is that we may improve this to roughly $n^{-n/d}$ under certain stronger demands about the cycles.

\begin{proof}
Define $d_1, d_2, d_3$ as follows.
\begin{enumerate}
    \item Let $d_1 = \max D_m$.
    \item Let $d_2$ be the largest element of $D_m$ which does not divide $6 d_1$. If there is no such $d_2$ then let $d_2 = 0$.
    \item Let $d_3$ be the largest element of $D_m$ which does not divide $6 d_1$ or $6 d_2$. If there is no such $d_3$ then let $d_3 = 0$.
\end{enumerate}

We claim that
\begin{equation}\label{eqn:d3-bound}
  d_3 \leq \frac49 \log_q m.
\end{equation}
Indeed, assume $d_2, d_3 > 0$. Every $d \in D_m$ either divides $6 d_1$, divides $6 d_2$, or is at most $d_3$. Moreover, $m$ is divisible by $\lcm(q^{d_1}-1, q^{d_2}-1, q^{d_3} - 1)$. Since
\begin{align*}
  \gcd(q^{d_1} - 1, q^{d_2} - 1) = q^{\gcd(d_1, d_2)} - 1 \leq q^{d_2 / 4} - 1, \\
  \gcd(q^{d_1} - 1, q^{d_3} - 1) = q^{\gcd(d_1, d_3)} - 1 \leq q^{d_3 / 4} - 1, \\
  \gcd(q^{d_2} - 1, q^{d_3} - 1) = q^{\gcd(d_2, d_3)} - 1 \leq q^{d_3/4} - 1,
\end{align*}
we have
\begin{align*}
  m
  &\geq \lcm(q^{d_1}-1, q^{d_2}-1, q^{d_3}-1) \\
  &\geq \frac{(q^{d_1} - 1)(q^{d_2}-1)(q^{d_3}-1)}{(q^{d_2/4} - 1)(q^{d_3/4} - 1) (q^{d_3 / 4} - 1)} \\
  &\geq q^{d_1 + 3d_2/4 + d_3/2} \\
  &\geq q^{(9/4)d_3}.
\end{align*}
The bound \eqref{eqn:d3-bound} follows.

Now let $p_n$ be be the probability that every cycle length of $\pi \in S_n$ is contained in $D_m$. Then $p_n$ is the coefficient of $z^n$ in
\[
  \prod_{d \in D_m} \sum_{c = 0}^\infty \frac{1}{c!} \pfrac{z^d}{d}^c = \exp \sum_{d \in D_m} \frac{z^d}{d},
\]
so for $r > 0$ we have
\[
  p_n \leq r^{-n} \exp \sum_{d \in D_m} \frac{r^d}{d}.
\]
Hence
\begin{align*}
  p_n
  &\leq r^{-n} \exp \left( \sum_{\substack{d \mid 6d_1 \\ d \leq d_1}} \frac{r^d}{d} + \sum_{\substack{d \mid 6d_2 \\ d \leq d_1}} \frac{r^d}{d} + \sum_{d \leq d_3} \frac{r^d}{d} \right).
\end{align*}
Assume $r \geq 1$. Then since only $O(1)$ divisors of $6d_1$ are at least $d_1/10$, we have
\[
  \sum_{\substack{d \mid 6d_1 \\ d \leq d_1}} \frac{r^d}{d}
  \leq \sum_{\substack{d \mid 6d_1 \\ d_1/10 < d \leq d_1}} \frac{r^d}{d}
  + \sum_{\substack{d \leq d_1/10}} \frac{r^d}{d}
  \leq O(r^{d_1}/d_1 + r^{d_1/ 10} \log d_1),
\]
and similarly for $d_2$, while
\[
  \sum_{d \leq d_3} \frac{r^d}{d} \leq O(r^{d_3} \log d_3).
\]
Hence
\[
  p_n \leq r^{-n} \exp O\left( r^{d_1}/d_1 + r^{d_1/10} \log d_1 + r^{d_3} \log d_3 \right).
\]

Put $r = n^{1/\log_q(m+1)}$. Since $d_1 \leq \log_q (m+1)$, and since $r^x/x$ has a unique local minimum, we have
\[
  r^{d_1} / d_1 \leq r^{\log_q(m+1)} / \log_q(m+1) + r = O(n/\log_q(m+1) + n^{1/\log_q(m+1)}).
\]
The first term dominates unless $\log_q(m+1)$ is at least comparable to $n$, so
\[
  r^{d_1} / d_1 = O(n / \log_q(m+1) + 1).
\]
Since $d_3 \leq \frac49 \log_q m$, we get
\[
  p_n \leq 
  n^{-n/\log_q (m+1)} \exp O(n / \log_q (m+1) + n^{1/10+o(1)} + n^{4/9 + o(1)}).
\]
The error terms here are even smaller than claimed.
\end{proof}

\subsection{Bounding \texorpdfstring{$\harm_1$}{eta\_1}} \label{subsec:harm_1}

Write $d_i = \deg f_i$, and write $c_d$ for the number of $i$ with $d_i = d$. Then
\[
  \harm_1
  = \sum_{(d_1, \dots, d_l) \vdash n}
  \frac1{c_1! \cdots c_n!} \sum_{f_1 \in \frakF_{d_1}} \cdots \sum_{f_l \in \frakF_{d_l}} \frac1{\lcm(\ord f_1, \dots, \ord f_l)}.
\]
If $f \in \frakF_d$, then $\ord f = \ord x$ for any root $x$ of $f$. Note that $x \in \F_{q^d}^\times$, and the map which sends each $x\in \F_{q^d}^\times$ of degree $d$ to its minimal polynomial is $d$-to-$1$. Thus rewriting the sum over $f \in \frakF_d$ as a sum over $x \in \F_{q^d}^\times$ of degree $d$, we get
\begin{align*}
  \harm_1
  &\leq \sum_{(d_1, \dots, d_l) \vdash n}
  \frac1{c_1! \cdots c_n! d_1 \dots d_l}
  \sum_{x_1 \in \F_{q^{d_1}}^\times} \cdots \sum_{x_l \in \F_{q^{d_l}}^\times} \frac1{\lcm(\ord x_1, \dots, \ord x_l)} \\
  &=\sum_{(d_1, \dots, d_l) \vdash n}
  \frac1{c_1! \cdots c_n! d_1 \dots d_l} |A| \harm_A,
\end{align*}
where $A$ is the abelian group
\[
  A = C_{q^{d_1} - 1} \times \cdots \times C_{q^{d_l} - 1}.
\]
Write
\begin{align*}
  a_1 &= q^{d_1}-1, \\
  a_2 &= (q^{d_2}-1) / \gcd(a_1, q^{d_1}-1), \\
  a_3 &= (q^{d_3}-1) / \gcd(a_1a_2, q^{d_2} - 1), \\
  &\vdots \\
  a_l &= (q^{d_l}-1) / \gcd(a_1 \cdots a_{l-1}, q^{d_l}-1)
\end{align*}
(these are the coprime parts of $q^{d_i} - 1$). Then there is a homomorphism of $A$ onto
\[
  B = C_{a_1} \times \cdots \times C_{a_l} \cong C_{\lcm(q^{d_1} - 1, \dots, q^{d_l}-1)}.
\]
Thus we deduce from Lemmas~\ref{quotient-lemma} and \ref{cyclic-group} that
\begin{align*}
  \harm_1
  &\leq q^n \sum_{(d_1, \dots, d_l)\vdash n} \frac1{c_1! \cdots c_n! d_1 \dots d_l} \frac{\divisor(\lcm(q^{d_1}-1, \dots, q^{d_l}-1))}{\lcm(q^{d_1}-1, \dots, q^{d_l} - 1)} \\
  &\leq q^n \sum_{m\in L} \frac{\divisor(m)}{m} \sum_{\substack{(d_1, \dots, d_l)\vdash n \\ q^{d_i} - 1 \mid m}} \frac1{c_1! \cdots c_n! d_1 \dots d_l} ,
\end{align*}
where $L$ is the set of all possible $\lcm$'s
\[
  m = \lcm(q^{d_1} - 1, \dots, q^{d_l} - 1) \qquad ((d_1, \dots, d_l) \vdash n).
\]

We recognize the inner sum here as the object of Lemma~\ref{nonsequitur}. Thus
\[
  \harm_1 \leq q^n \sum_{m \in L} \frac{\divisor(m)}{m} n^{-n/\log_q (m+1)} e^{O(n/\log_q m + n^{1/2})}.
\]
Applying the divisor bound \eqref{divisor-bound}, we get
\[
  \harm_1 \leq q^n \sum_{m \in L}
  \exp\left(
    -\log m - \frac{n \log n \log q}{\log (m+1)} + O\left(\frac{\log m}{\loglog m} + \frac{n \log q}{\log m} + n^{1/2}\right)
  \right).
\]
It is easy to see that the maximum occurs at $\log m \approx \sqrt{n \log n \log q}$. Since by \eqref{partition-bound} we have
\[
  |L| \leq p(n) \leq e^{O(n^{1/2})},
\]
we deduce that
\begin{equation} \label{harm_1-bound}
  \harm_1 \leq q^n \exp\left(
  -2\sqrt{n \log n \log q}
  + O\left(
      \sqrt{n \log q}
    \right)
  \right).
\end{equation}

\subsection{Bounding \texorpdfstring{$\harm_2$}{eta\_2}}

We have to choose a multiset of pairs $(f_i, m_i)$, $f_i \in \frakF$, $m_i \geq 2$, such that $\sum m_i \deg f_i = n$. Let $\mu_i = m_i \deg f_i$, and write $c_\mu$ for the number of $i$ with $\mu_i = \mu$. Then, since $|\frakF_d| \leq q^d/d$,
\[
  \harm_2 \leq \sum_{(\mu_1, \dots, \mu_l) \vdash n} \frac1{c_1! \cdots c_n!} \prod_{i=1}^l \sum_{d_i \mid \mu_i, d_i < \mu_i} \frac{q^{d_i}}{d_i}.
\]
Since $q^d / d$ is increasing in $d$ for $q \geq 2$, $d \geq 1$, and since $\divisor(x) \leq x/2 + 1$ for all $x$, we have
\begin{align*}
  \harm_2
  &\leq \sum_{(\mu_1, \dots, \mu_l) \vdash n}
  \frac1{c_1! \cdots c_n!}
  \prod_{i=1}^l \sum_{d_i \mid \mu_i, d_i < \mu_i}
  \frac{q^{\mu_i/2}}{\mu_i/2} \\
  &\leq \sum_{(\mu_1, \dots, \mu_l) \vdash n}
  \frac1{c_1! \cdots c_n!}
  \prod_{i=1}^l q^{\mu_i/2} \\
  &\leq p(n) \, q^{n/2}.
\end{align*}
Thus by \eqref{partition-bound},
\begin{equation} \label{harm_2-bound}
  \harm_2 \leq e^{O(n^{1/2})} q^{n/2}.
\end{equation}

We can now deduce Theorem~\ref{inverse-orders-theorem}, in the case $G = \GL(n,q)$. From \eqref{harm-splitting}, \eqref{harm_1-bound}, and \eqref{harm_2-bound} we have
\[
  \harm_G \leq \sum_{n_1 + n_2 = n} \exp\left(- 2\sqrt{n_1 \log n_1 \log q} - \frac12 n_2 \log q + O(\sqrt{n \log q}) \right)
\]
Split the summation according to whether $n_2 \geq \sqrt{n} \log n$, say. Note that if $n_2 \leq \sqrt{n} \log n$ then, since $\sqrt{x \log x}$ has derivative comparable to $x^{-1/2} (\log x)^{1/2}$, we have
\[
  2\sqrt{n_1 \log n_1}
  = 2\sqrt{n \log n} - O((\log n)^{3/2}).
\]
The claimed bound follows.

\subsection{The groups \texorpdfstring{$\PGL(n, q)$, $\SL(n,q)$, and $\PSL(n, q)$}{PGL(n,q), SL(n,q), and PSL(n,q)}}

When $n$ is large compared to $\log q$, the cases of $\PGL(n,q)$, $\SL(n,q)$, and $\PSL(n,q)$ follow from the case of $\GL(n,q)$ and Lemmas~\ref{subgroup-lemma} and \ref{quotient-lemma}, because a factor of $q$ is negligible. But when $n$ is at most comparable to $\log q$ then we cannot be blas\'e about factors of $q$, so we have to review the proof.

Consider first the case of $\PGL(n, q)$. We can repeat the analysis of Subsection~\ref{subsec:harm_1}, now measuring order in $\PGL$. We need to modify the analysis only for the discrete partition $(1, \dots, 1)$. But that term is just the sum, over all $x_1, \dots, x_n \in \F_q^\times$, of the projective order of the diagonal matrix with entries $x_1, \dots, x_n$, which is exactly $(q-1)^n \harm_A$ for $A = {\F_q^\times}^{n-1}$. By Lemmas~\ref{quotient-lemma} and \ref{cyclic-group} this is bounded by $q^n \divisor(q-1) / (q-1)$, provided of course that $n > 1$. As a result we find that
\[
  \eta_1 \leq q^n e^{O(n^{1/2})} \frac{\divisor(q-1)}{q},
\]
which is more than sufficient to finish the proof.

Now consider $\SL(n, q)$. Again we can repeat the analysis of Subsection~\ref{subsec:harm_1}, now restricting to $g$ with $\det g = 1$. Again we need to reconsider the discrete partition $(1, \dots, 1)$ in Subsection~\ref{subsec:harm_1}. In this case that term is
\[
  \sum_{\substack{x_1, \dots, x_n \in \F_q^\times \\ x_1 \dots x_n = 1}} \frac1{\lcm(\ord_q x_1, \dots, \ord_q x_n)},
\]
which we again recognize as $(q-1)^{n-1} \harm_A$ for $A = {\F_q^\times}^{n-1}$ provided that $n > 1$, and we continue as before.

Finally, note $\harm_{\PSL(n,q)} \leq 2 \, \harm_{\SL(n,q)}$, and we can continue to be blas\'e about factors of $2$. This proves the remaining cases of Theorem~\ref{inverse-orders-theorem}.

\section{Using character bounds}

Following \cite{liebeck--shalev--99, larsen-shalev-tiep}, the \emph{support} of an element $g \in \GL(n,q)$ is defined to be the codimension of the largest eigenspace of $\pi$:
\def\supp{\op{supp}}
\[
  \supp g = \min_{\lambda \in \overline{\F}_q} \op{codim} \op{ker} (g - \lambda).
\]
Note that if $\supp g < n/2$ then $\lambda \in \F_q$. We will use the following bound from \cite{larsen-shalev-tiep}:

\begin{theorem}[Larsen--Shalev--Tiep]\label{larsen-shalev-tiep-bound}
Let $G = \SL(n,q)$. If $\chi\in\Irr(G)$ and $\chi \neq 1$ then, for all $g \in G$,
\[
  \frac{|\chi(g)|}{\chi(1)} \leq q^{-\sqrt{|\supp{g}|} / 481}.
\]
\end{theorem}

We will use this to bound $\langle \chi, \eta \rangle / \chi(1)$.

\begin{theorem} \label{thm:chietabound}
We have the following bound for $|\langle \chi, \eta\rangle | / \chi(1)$ when $G = \SL(n, q)$:
\[
  \chietaratio \leq q^{-c \sqrt{n}} + e^{-(2-o(1)) \sqrt{n \log n \log q}}.
\]
\end{theorem}
\begin{proof}
By Theorem~\ref{larsen-shalev-tiep-bound}, for $\chi\neq 1$ we have
\begin{align}
  \chietaratio
  &\leq q^{-c \sqrt{n}} + \frac1{|G|} \sum_{\substack{g \in G \\ \supp g < n/2}} \eta(g) \nonumber \\
  &= q^{-c\sqrt{n}} + \frac1{|G|} \sum_{\substack{g \in G \\ \supp g < n/2}} \sum_{\substack{\sigma \in G \\ g \in \langle \sigma \rangle}} \frac1{\ord \sigma}. \label{first-chi-eta-bound}
\end{align}

Fix $g \in G$ with $\supp g < n/2$, and let $U$ be the large eigenspace of $g$. Write $G_U$ for the set-wise stabilizer of $U$ in $G$, and write $p_U$ for the map $p_U:G_U \to \PGL(U)$. Note that if $g \in \langle \sigma \rangle$ then $\sigma \in G_U$ (because $\sigma$ preserves the eigenspaces of $\sigma^i$).
Hence
\begin{align*}
  \frac1{|G|} \sum_{\substack{g \in G \\ \supp g < n/2}} \sum_{\substack{ \sigma \in G \\ g \in \langle \sigma \rangle}} \frac1{\ord \sigma}
  &\leq \sum_{\substack{U \leq \F_q^n \\ \dim U > n/2}} \frac1{|G|} \sum_{\substack{g \in G_U \\ p_U(g)=1}} \sum_{\substack{\sigma \in G_U \\ g \in \langle \sigma \rangle}} \frac1{\ord \sigma} \\
  &= \sum_{\substack{U \leq \F_q^n \\ \dim U > n/2}} \frac1{|G|} \sum_{\sigma \in G_U} \frac{1}{\ord \sigma} \sum_{\substack{g \in \langle \sigma \rangle \\ p_U(g)=1}} 1 \\
  &= \sum_{\substack{U \leq \F_q^n \\ \dim U > n/2}} \frac1{|G|} \sum_{\sigma \in G_U} \frac{1}{\ord p_U(\sigma)} \\
  &= \sum_{\substack{U \leq \F_q^n \\ \dim U > n/2}} \frac{1}{[G:G_U]} \, \harm_{p_U(G_U)}.
\end{align*}
Now note that $G$ acts transitively on the subspaces of each dimension, and
\[
  p_U(G_U) \cong
  \begin{cases}
    \PGL(U) & \text{if}~\dim U<n, \\
    \PSL(U) & \text{if}~\dim U=n.
  \end{cases}
\]
Thus
\[
  \chietaratio \leq q^{-c\sqrt{n}} + \sum_{n/2 < d < n} \eta_{\PGL(d, q)} + \eta_{\PSL(n,q)},
\]
so the theorem follows from Theorem~\ref{inverse-orders-theorem}.
\end{proof}

\section{A choice of sets \texorpdfstring{$\frakC$}{C}}

In this section we define large conjugation-invariant subsets $\frakC_1, \frakC_2 \subset \SL(n,q)$, and prove that if $g_1 \in \frakC_1$ and $g_2 \in \frakC_2$ then $\langle g_1, g_2 \rangle = \SL(n,q)$. Theorem~\ref{main-thm} will then follow from Theorems~\ref{thm:pisigma} and \ref{thm:chietabound}.

\begin{enumerate}
    \item Let $\frakC_1 \subset \SL(n,q)$ be the set of all irreducible $g$ of order $(q^n-1)/(q-1)$.
    \item Let $\frakC_2 \subset \SL(n,q)$ be the set of all $g$ of order $q^{n-1}-1$ splitting $\F_q^n$ as $\ell \oplus W$, where $\dim \ell=1$ and $\dim W = n-1$.
\end{enumerate}

\begin{lemma} \label{size}
For $i\in \{1,2\}$ we have
\[
  \frac{|\frakC_i|}{|\SL(n,q)|} \gg \frac1{n \loglog(q^n)} .
\]
\end{lemma}
\begin{proof}
Let $G = \GL(n,q)$. Consider first $\frakC_1$. For every $g \in \frakC_1$ we have $C_G(g) \cong \F_{q^n}^\times$, and the $G$-conjugacy classes in $\frakC_1$ are in bijection with irreducible polynomials of degree $n$ and order $(q^n - 1)/(q-1)$. Thus
\[
  \frac{|\frakC_1|}{|G|} = \frac{\totient((q^n-1)/(q-1))}{n (q^n - 1)}.
\]
The claimed estimate follows from \eqref{totient-bound}.

Now consider $\frakC_2$. In this case, the centralizers are isomorphic to $\F_{q^{n-1}}^\times \times \F_q^\times$, and conjugacy classes are in bijection with irreducible polynomials of degree $n-1$ and order $q^{n-1} - 1$ (unless $(n, q) = (2,3)$). There are $\totient(q^{n-1} - 1) / (n-1)$ such polynomials. The claim follows as before.
\end{proof}

\begin{lemma}
Assume $n \geq 3$ and $(n,q) \neq (3,4)$. Let $g_1 \in \frakC_1$ and $g_2 \in \frakC_2$. Then $\langle g_1, g_2 \rangle = \SL(n,q)$.
\end{lemma}
\begin{proof}
Let $G = \langle g_1, g_2 \rangle$.

\begin{claim}
$G$ acts transitively on $\F_q^n \setminus\{0\}$.
\end{claim}

We may assume that $\F_q^n$ is identified with $\F_{q^n}$ in such a way that $g_1$ acts as some $\alpha\in\F_{q^n}^\times$ of order $(q^n-1)/(q-1)$. Then the orbits of $g_1$ are precisely the fibres of the norm map. On the other hand let $\F_q^n = \ell \oplus W$ be the splitting respected by $g_2$. Then $g_2$ acts transitively on the nonzero points of $W$. Thus it suffices to prove that the norm map is surjective on $W$. This is proved in the appendix: see Corollary~\ref{cor:sawin-hyperplane}.

\begin{claim}
$G$ acts transitively on planes in $\F_q^n$.
\end{claim}

Again let $\F_q^n = \ell \oplus W$ be the splitting respected by $g_2$. Since by the previous claim we know that $G$ is transitive on lines through the origin, it suffices to prove that the stabilizer $G_\ell$ is transitive on planes containing $\ell$. But planes containing $\ell$ are in obvious bijection with lines through the origin in $W$, and $g_2$ acts transitively on the nonzero points of $W$, so we're done.

\begin{claim}
$G = \SL(n,q)$ (i.e., the lemma holds).
\end{claim}

Linear groups which are transitive on $k$-flats for some $k$ in the range $2 \leq k \leq n-2$ have been classified by Cameron and Kantor~\cite{cameron-kantor, cameron-kantor-2018}. Applying \cite[Proposition~8.4]{cameron-kantor-2018}, we find that either $G = \SL(n,q)$ or $G \cong A_7$ inside $\SL(4,2)$. The latter group is obviously ruled out (having no element of order $15$ for instance), so $G = \SL(n,q)$.

\end{proof}

We can finally prove Theorem~\ref{main-thm}.

\begin{proof}[Proof of Theorem~\ref{main-thm}]
The case $n=2$ is classical, and follows from Dickson's classification of subgroups of $\SL(2, q)$ (see Suzuki~\cite[Section~3.6]{suzuki} for a modern treatment), so assume $n\geq 3$.

Let $G = \langle \pi, \sigma\rangle$. Combining Theorem~\ref{thm:pisigma}, Theorem~\ref{thm:chietabound}, and Lemma~\ref{size}, for $i \in \{1, 2\}$ we have
\[
  \P(G \cap \frakC_i = \emptyset ) \leq n \loglog (q^n) \left( q^{-c\sqrt{n}} + e^{-(2-o(1)) \sqrt{n \log n \log q}} \right).
\]
As long as either $q$ or $n$ is large, the factor $n \loglog(q^n)$ can be absorbed into the $c$ or the $o(1)$ (and the theorem is vacuous if both $q$ and $n$ are bounded). Thus almost surely $G$ contains a $g_1 \in \frakC_1$ and a $g_2 \in \frakC_2$, so $G = \SL(n,q)$ by the previous lemma.
\end{proof}

\appendix

\section{Surjectivity of the norm on subspaces of \texorpdfstring{$\F_{q^n}$}{F\_q\^{}n}}

\def\norm{N}
\def\C{\mathbf{C}}

The norm of the extension $\F_{q^n} / \F_q$ is the map $\norm: \F_{q^n} \to \F_q$ defined by
\[
  \norm(x) = x x^q \cdots x^{q^{n-1}} = x^{(q^n-1)/(q-1)}.
\]
It is a basic fact of finite fields that $\norm$ is surjective. It turns out that $\norm$ remains surjective on any sufficiently large subspace.

\begin{theorem}\label{sawin-lemma}
Let $W \leq \F_{q^n}$ be an $\F_q$-linear subspace, and suppose either
\begin{enumerate}
    \item $\dim W \geq n/2 + 1$, or
    \item $\dim W = (n+1)/2$ and $\gcd(n, q-1) < q^{1/2} + 1$.
\end{enumerate}
Then $\norm$ is surjective on $W$.
\end{theorem}

We are grateful to Sawin~\cite{sawin-MO} for the following proof. The wording has been modified slightly, only to spell out the details a little more.

\begin{proof}
The number of elements of $W$ with norm $a$ is
\[
  \frac{1}{q-1} \sum_{\chi: \F_q^\times \to \C^\times} \sum_{x \in W} \chi(\norm(x)) \overline{\chi(a)}.
\]
The summand corresponding to $\chi$ vanishes unless $\chi$ has order dividing $n$ (as
\[
  \chi(\norm(\lambda x)) = \chi(\lambda)^n \chi(\norm(x))
\]
for $\lambda \in \F_q$), so there are at most ${\gcd(n,q-1)}$ nonzero terms. The $\chi=1$ term has size $q^{\dim W}$, so it is sufficient to prove that each of the other terms has size less than
\[
  q^{\dim W} / (\gcd (n,q-1)-1).
\]

Note that
\begin{equation}\label{sawin-term}
  \sum_{x \in W} \chi(\norm(x)) = \sum_{x \in \F_{q^n}} \frac1{|W^\perp|} \sum_{\theta \in W^\perp} \theta(x) \chi(\norm(x)),
\end{equation}
where $W^\perp$ is the set of additive characters $\theta:\F_{q^n} \to \C^\times$ which vanish on $W$. The sum
\[
  \sum_{x \in \F_{q^n}} \theta(x) \chi(\norm(x))
\]
is a Gauss sum, and thus has size $q^{n/2}$ for $\chi \neq 1$. We deduce that \eqref{sawin-term} also has size at most $q^{n/2}$. Thus we are done provided that
\[
  q^{\dim W - n/2} > \gcd(n, q-1) - 1,
\]
and this is easily checked under either hypothesis.
\end{proof}

\begin{remark}
If $n$ is even there is a subspace $W = \F_{q^{n/2}}$ of dimension $n/2$ on which $\norm(x)$ is always a square, and in particular not surjective. We do not know any counterexamples with $\dim W = (n+1)/2$.
\end{remark}

We actually only need the $\dim W = n-1$ case, which we single out now.

\begin{corollary}\label{cor:sawin-hyperplane}
Let $n\geq 3$, and let $W \leq \F_{q^n}$ be an $\F_q$-linear hyperplane. Then $\norm$ is surjective on $W$.
\end{corollary}

\begin{proof}
This is immediate from the theorem, with the exception of $(n, q) = (3,4)$, and the case $(n, q) = (3, 4)$ can be checked directly.
\end{proof}

\bibliographystyle{alpha}
\bibliography{refs}

\end{document}